\documentclass[12pt]{amsart} 

 \usepackage{caption}
 \usepackage{subcaption}

\usepackage[dvipsnames,usenames]{xcolor} 

\usepackage{amsfonts,latexsym,wasysym, mathrsfs, mathtools,hhline,color, bm}
\usepackage[all, cmtip]{xy}

\usepackage{url}

\definecolor{hot}{RGB}{65,105,225} 

\usepackage[pagebackref=true,colorlinks=true, linkcolor=hot ,  citecolor=hot, urlcolor=hot]{hyperref}

\usepackage{textcomp}
\usepackage{tipa}

\usepackage{graphicx,enumerate}
\usepackage{enumitem}
\usepackage[normalem]{ulem}

 \usepackage{amsthm}  
 \usepackage{amsmath} 
 \usepackage{amsfonts}  
 \usepackage{amssymb}
 \usepackage{amscd} 
 \usepackage[all]{xy} 

\usepackage{graphicx}
\usepackage{pdfsync}
\usepackage{url}
\usepackage{comment}

\usepackage[linesnumbered,ruled,vlined]{algorithm2e}
\usepackage{cleveref}

\newtheorem{thm}{Theorem}[section]

\theoremstyle{definition}
\newtheorem{example}[thm]{Example}
\newtheorem{theorem}[thm]{Theorem}
\newtheorem{lem}[thm]{Lemma}
\newtheorem{prop}[thm]{Proposition}
\newtheorem{cor}[thm]{Corollary}
\theoremstyle{definition}
\newtheorem{defn}[thm]{Definition}

\newtheorem{rem}[thm]{Remark}
\newtheorem{notation}[thm]{Notation}
\numberwithin{equation}{section}

\newcommand{\PP}{\mathbb{P}}

\newcommand{\CC}{\mathbb{C}}

\newcommand{\sym}{{\mathrm{Sym}}}

\newcommand{\gr}{{\mathrm{GR}}}
\newcommand{\sgr}{{\mathrm{SGR}}}
\newcommand{\codim}{{\mathrm{codim}}}

\title{The symmetric geometric rank of symmetric tensors}

\author[Lindberg]{Julia Lindberg}
\address{
Julia Lindberg \\
University of Texas-Austin \\ Austin TX, USA
}
\email{julia.lindberg@math.utexas.edu}
\urladdr{\url{https://sites.google.com/view/julialindberg/home}}

\author[Santarsiero]{Pierpaola Santarsiero}
\address{
 Pierpaola Santarsiero\\
Osnabrück University and MPI MiS \\ Germany
}
\email{pierpaola.santarsiero@uni-osnabrueck.de}
\urladdr{\url{https://pierpaolasantarsiero.wixsite.com/pierpaola}}

\begin{document}

\maketitle

\begin{abstract}
Inspired by recent work of Kopparty-Moshkovitz-Zuiddam and motivated by problems in combinatorics and hypergraphs, we introduce the notion of the symmetric geometric rank of a symmetric tensor. This quantity is equal to the codimension of the singular locus of the hypersurface associated to the tensor. We first derive fundamental properties of the symmetric geometric rank. Then, we study the space of symmetric tensors of prescribed symmetric geometric rank, which are spaces of homogeneous polynomials whose corresponding hypersurfaces have a singular locus of bounded codimension.
\end{abstract}

\section{Introduction}
In the last few decades, several notions of tensor rank have arisen and decomposing a tensor in terms of a particular notion of rank has become a useful tool in recovering hidden information in the data of a given tensor  \cite{ABMM,BBM,BT,blasiak2017cap,BBdiff,BBapolarity,CFTZsymsub,CGJ19,CGLM,derksgstable,GOV19,gowers2011linear,strassensubrank,tao2016symmetric}.
		
Recently, a new notion of rank for
tensors was introduced in \cite{kopparty2020geometric}. Given a tensor $T=(t_{i,j,k}) \in \mathbb{C}^{n_1+1}\otimes \mathbb{C}^{n_2+1}\otimes \mathbb{C}^{n_3+1}$, for $k=0,\dots,n_3$ denote $T_k=(t_{i,j,k})_{i,j}\in \mathbb{C}^{n_1+1}\otimes \mathbb{C}^{n_2+1}$ the $n_3+1$ matrices obtained by slicing the tensor with respect to the third factor. The authors define the \emph{geometric rank} $\mathrm{GR}(T)$ of $T$ in terms of the codimension of an affine algebraic variety. Namely,
\[
\mathrm{GR}(T):=\codim \{ (x,y) \in \mathbb{C}^{n_1+1}\times \mathbb{C}^{n_2+1} ~\vert ~ x^TT_0y=\cdots = x^TT_{n_3}y=0 \}.
\] 
In \cite[Theorem 3.2]{kopparty2020geometric} the authors proved that the definition of geometric rank is invariant with respect to any of the three the slicings. 
This notion of rank can be naturally extended for an arbitrary number of factors (cf. \cite[Section 2]{kopparty2020geometric}) and was introduced as a new tool for studying tensors and hypergraphs. A geometric study of this notion of rank was considered in \cite{GL22}, where the authors made a classification of tensors with geometric rank one and two. In a subsequent work \cite{G22}, the first author of \cite{GL22} classified tensors with geometric rank three.

A natural question is whether, as with other notions of tensor rank, there is a symmetric version of the geometric rank of a symmetric tensor. 
In this paper we address this problem by introducing the notion of the \emph{symmetric geometric rank} of a symmetric tensor. More precisely, let $T=(t_{i,j,k})\in \sym^3\mathbb{C}^{n+1}$ be a symmetric tensor, and as before, denote $T_k=(t_{i,j,k})_{i,j}$ for $k=0,\dots,n$. 
For a symmetric tensor $T\in \sym^3\mathbb{C}^{n+1}$, 
we define the \emph{symmetric geometric rank} of $T$ as
\begin{equation}\label{eq:symgr}
\sgr(T) = \codim \{ x \in \mathbb{C}^{n+1} ~\vert ~ x^TT_0x=\cdots =x^TT_{n}x=0 \}.
\end{equation}
A tensor $T\in \sym^3\mathbb{C}^{n+1}$ can naturally be seen as a degree three homogeneous polynomial $F \in \CC[x_0,\ldots,x_{n}]$, for example by considering the $(i,j,k)$th entry of the tensor as the coefficient of $x_{i}x_jx_k $, and throughout the rest of this paper we interchangeably consider a tensor both as an element $T \in \sym^3\CC^{n+1}$ or as a homogeneous polynomial $F \in \CC[x_0,\ldots,x_{n}]_3$.

By considering an order three tensor as a homogeneous cubic polynomial, it is easy to see that for all $k=0,\dots,n$ the equation $x^TT_kx$=0 in \eqref{eq:symgr}  corresponds to considering 
$\frac{\partial F}{\partial x_k}=0$  
and hence when $F$ is the cubic polynomial corresponding to a tensor $T$, \eqref{eq:symgr} is equivalent to
\[
\codim \left\{
x \in \mathbb{C}^{n+1} ~\middle|~
\frac{ \partial F}{\partial x_0} (x) = \dots =  \frac{ \partial F}{\partial x_{n}}(x) =0 
\right\}.
\]

By passing to the projective setting, we recognize that the above codimension is actually the codimension in $\mathbb{P}^n$ of the singular locus, $\mathrm{Sing}(F)$, of the (possibly non-reduced) hypersurface in $\mathbb{P}^{n}$ defined by $\{ [x]\in \mathbb{P}^n ~\vert~ F(x)=0\}$. 
Thus we define the symmetric geometric rank of an order $d$ tensor as follows.
\begin{defn}Let $[T]\in \mathbb{P}(\sym^d\mathbb{C}^{n+1})$ and let $F\in \mathbb{C}[x_0,\dots,x_n]_d$ be the homogeneous degree $d$ polynomial associated to $F$. The \emph{symmetric geometric rank} of $T$~is
	\begin{align*}
	\sgr(T):&=\codim (\mathrm{Sing}(F)) \\[0.1em]
&= \codim \left\{
x \in \mathbb{P}^{n} ~\middle|~
\frac{ \partial F}{\partial x_0} (x) = \dots =  \frac{ \partial F}{\partial x_{n}}(x) =0 
\right\},
\end{align*}
 where we emphasize that we are considering the codimension of $\mathrm{Sing}(F) $ in $\mathbb{P}^n$, not $V(F)$. We also emphasize that in contrast to \cite{kopparty2020geometric}, we chose to work with projective spaces instead of affine spaces, but the notion of symmetric geometric rank is the same whether one works affinely or projectively. 
\end{defn} 
It is clear that the notion of symmetric geometric rank is invariant with respect to any choice of the flattening.

We find interesting that the singular locus of a hypersurface, which is a classical object in algebraic geometry, appears in the context of tensors. This was also highlighted in \cite{LTranksym}, where the authors gave a bound on the symmetric rank in terms of dimensions of singular loci   \cite[Theorem 1.3]{LTranksym}. In addition, a detailed study on the rank of cubic hypersurfaces according to the classification of their singular loci was done in \cite{SS20}. 

\subsection*{Symmetric geometric rank as a useful tool for hypergraphs}
We remark that the notion of the symmetric geometric rank of a symmetric tensor is relevant for the same applications that inspired the definition of geometric rank in \cite{kopparty2020geometric}.
This is because many of the motivating applications in combinatorics consider a highly structured tensor, which is symmetric.
More precisely, one of these applications 
comes from hypergraph theory (\cite[Section 1.3]{kopparty2020geometric}). We recall that an (undirected uniform) \emph{hypergraph} is a couple $(V,E)$, where $V=\{1,\dots,n\}$ is a finite set collecting the \emph{vertices} of the hypergraph, while $E\subset 2^V$ contains the \emph{edges} and we require that all edges $e\in E$ have the same cardinality $d$. 
One interesting problem is to compute the \emph{independence number} of the hypergraph, which is the cardinality of the largest set of vertices containing no edges of $(V,E)$. One way to bound such a quantity is to compute the subrank $\mathrm{Q}(T)$ of the associated tensor $T$, which is the maximum identity tensor ``contained'' in $T$. 
Indeed, we remark that given a hypergraph $(V,E)$ as above, we can associate a tensor in $ (\mathbb{C}^{n})^{\otimes d}$ to the hypergraph as follows. We define $T=(t_{i_1,\dots,i_d})$ as
\begin{equation}\label{eq:tensor_hypergraph}
t_{i_1,\dots,i_d}:= \begin{cases}
1, & \hbox{ if } \{i_1,\dots,i_d\}\in E\\
0, &\hbox{otherwise.}
\end{cases}
\end{equation}
Computing the subrank of the associated tensor is computationally intractable in many cases. Therefore, since the geometric rank is an upper bound for the subrank, it can be seen as a new tool to bound this quantity. 

However, an obvious observation is that the tensor defined in \eqref{eq:tensor_hypergraph} is symmetric and hence it is desirable to exploit its structure. Taking advantage of the symmetric structure of the tensor associated to a hypergraph is also one of the motivations for the introduction of the symmetric subrank in \cite[Section 1, Section 2.2]{CFTZsymsub}. Unfortunately, as with the notion of subrank, the symmetric subrank, $\mathrm{Q}_s(T)$, of a tensor $T$ is hard to compute. Our new notion of rank is then useful to bound this quantity, since a consequence of the forthcoming \Cref{thm:bound_sym_subrank} is that, for a symmetric tensor $T$,
$$
\mathrm{Q}_s(T) \leq \sgr(T)\leq \gr(T).
$$
Even though the purpose of this manuscript is to explore geometric aspects of this new notion of rank, the above discussion motivates further investigation of the symmetric geometric rank as a tool to study the aforementioned types of problems.

\subsection*{Parametrizing tensors with respect to symmetric geometric rank}
With the introduction of any notion of rank, it is natural to consider spaces of tensors with that prescribed rank. We study these spaces for the symmetric geometric rank and stratify the symmetric tensor space by these spaces.

\begin{defn}\label{spaces_sgr_r} Fix a positive integer $d$ and let $r\geq 0$. The \emph{space of symmetric $d$-factors tensors with symmetric geometric rank at most $r$} is 
	$$
	\mathcal{S}_{d,r}:=\{ [T]\in \PP (\sym^d \CC^{n+1}) ~\vert ~ \sgr(T)\leq r   \}.
	$$
\end{defn} 

By considering a tensor $T \in \PP(\sym^d \CC^{n+1})$ as a homogeneous polynomial $F \in \PP(\sym^d \CC^{n+1})^*$, we can define $\mathcal{S}_{d,r}$ from the following incidence variety
\begin{align*}   
	\widetilde{Sing}=\{ (x, [F]) \, \vert \, V(F) \mbox{ is singular on $x$ } \}\subset \mathbb{P}^n \times \mathbb{P}^{{{n+d} \choose d}-1}.
\end{align*}
Denote $\pi_2\colon \widetilde{Sing} \rightarrow \mathbb{P}^{{{n+d} \choose d}-1}$ as the projection onto the second factor and $\pi_1$ as the projection onto $\mathbb{P}^n$. With the above geometric construction we can rephrase the space of tensors having prescribed symmetric geometric rank as
\begin{align*}
	\mathcal{S}_{d,r} &= \{ [F] \in \PP(\sym^d\mathbb{C}^{n+1})^* \, \vert \, \dim \mathrm{Sing}([F]) \geq n-r   \}\\
	&=\{ [F] \in \PP(\sym^d\mathbb{C}^{n+1})^* \, \vert \, \dim \pi_1(\pi_2^{-1}([F]))\geq n-r  \}.
\end{align*}

This construction via the incidence variety  $\widetilde{Sing}$ and the upper semicontinuity of the dimension of a fiber, shows that each $\mathcal{S}_{d,r}$ is a Zariski closed set in $\PP(\sym^d\mathbb{C}^{n+1})^*$. 
This then implies that the following chain of inclusions holds:
\begin{align*}\label{eq:skr_chain}
\PP(\sym^d\mathbb{C}^{n+1})=\mathcal{S}_{d,n+1} \supset \mathcal{S}_{d,n}\supset \cdots \supset \mathcal{S}_{d,1}\neq \emptyset,
\end{align*}
where we remark (and later prove) that $\mathcal{S}_{d,1}$ contains the $d$th Veronese variety $\nu_d(\mathbb{P}^n)$ so it is not empty. Moreover, we recognize that $\mathcal{S}_{d,n}$ is a classical object in algebraic geometry whose corresponding defining equation is known as the \emph{discriminant} \cite{GKZ}. 
Notice also that the smallest element in the above chain should actually be $\mathcal{S}_{d,0}$ which consists of classes of homogeneous polynomial whose corresponding hypersurface in $\mathbb{P}^n$ has an $n$-dimensional singular locus. The only such polynomial is the zero polynomial, so $\mathcal{S}_{d,0}=\emptyset$ and hence it will not play a role in our discussion.
\medskip

\subsection*{Outline of the paper} We begin in \Cref{sec:preliminary} by reviewing standard notions related to symmetric tensors and their geometry. In Section \ref{sec:properties} we prove basic properties of the symmetric geometric rank (\Cref{lem:monotone}, \Cref{lem:additive}, \Cref{lem:subadditive}) as well as compare this notion of rank with other commonly used definitions of rank.  Section ~\ref{sec:spacesprescribedSGR} is devoted to classifying spaces of symmetric tensors with prescribed symmetric geometric rank. We give a geometric interpretation of these spaces in \Cref{thm:full_S_dr} and we provide full classifications of $\mathcal{S}_{d,r}$ when $(d,r) \in \{(2,r), (d,1), (3,2) \}$ (\Cref{prop:sgr_mats}, \Cref{thm:S_{3,1}}, \Cref{thm:complete_S_k1} and \Cref{thm:S_{3,2}}). 
In the last section we compute the symmetric geometric rank of relevant tensors such as the Big and Small Coppersmith-Winogard tensor, the maximal compressibility tensor, and the symmetrized part of the matrix multiplication tensor (\Cref{ex:bigCW}, \Cref{ex:smallCW}, \Cref{ex:maxc} and \Cref{cor:sgr_sym_matr_mult}). We also give a concrete description of the decomposition of the symmetric tensor space via $\mathcal{S}_{d,r}$for $d=3$ in the particular cases of $n=1,2$.

\subsection*{Acknowledgements} We would like to thank Jose Rodriguez for his input and guidance, especially at the beginning of this project.
We are grateful to Daniele Agostini, Ángel Ríos and Jeroen Zuiddam for helpful conversations on the problems addressed in this paper. We also want to thank Andreas Kretschmer for pointing us to the result of \Cref{thm:full_S_dr}.
We acknowledge the Thematic Research Programme ``Tensors: geometry, complexity and quantum entanglement", University of Warsaw, Excellence Initiative – Research University and the Simons Foundation Award No. 663281 granted to the Institute of Mathematics of the Polish Academy of Sciences for the years 2021-2023.

Santarsiero was partially supported by the Deutsche Forschungsgemeinschaft (DFG, German Research Foundation) -- Projektnummer 445466444.

\section{Preliminaries}\label{sec:preliminary}
Throughout the rest of this paper we work over $\mathbb{C}$.

\begin{notation}\label{notation: convention}
	We denote $e_0,\dots,e_{n}$ as the canonical basis of $\CC^{n+1}$ and when representing a tensor $T\in (\CC^{n+1})^{\otimes d} $ in coordinate as $T=(t_{i_1,\dots,i_d})$, we will always consider it with respect to canonical basis on each factors. 
\end{notation}

\begin{notation}
Note that any tensor $T \in (\CC^{n+1})^{\otimes d}$ can be seen as a multilinear polynomial, $F$, in the variable groups $\{x_1 = \{x_{1,1},\ldots,x_{1,n+1}\},$ $\ldots ,x_d = \{x_{d,1},\ldots,x_{d,n+1}\}\}$. Further, symmetric tensors can be seen as homogeneous polynomials. In the rest of the paper we will treat elements in $\sym^d(\mathbb{C}^{n+1})$ both as symmetric tensors and as homogeneous polynomial of degree $d$, exploiting the isomorphism $\sym^d\mathbb{C}^{n+1}\cong (\sym^d\CC^{n+1})^*$ that sends $t_{i_1,\dots,i_d}$ to $t_{i_1,\dots,i_d} x_{i_1}\cdots x_{i_d}$. 
	\end{notation}

Let us start by recalling basic facts about symmetric tensors and their geometry.
\begin{defn}
	A symmetric tensor $T\in \sym^d\mathbb{C}^{n+1}$ is called \emph{elementary} if it can be written as the tensor product of the same vector $v\in \mathbb{C}^{n+1}$, i.e.
	$$
	T= \underbrace{v \otimes \cdots \otimes v}_{d \text{ times}}.
	$$ 
\end{defn}
Elementary symmetric tensors are the building blocks of a symmetric tensor decomposition and any symmetric tensor can be decomposed as a sum of symmetric elementary tensors.
\begin{defn} The minimum number of symmetric elementary tensors needed to decompose a tensor $T\in \sym^d\mathbb{C}^{n+1}$ is the \emph{symmetric rank} $r_s(T)$ of $T$, also called \emph{Waring rank}.
	\end{defn}
Symmetric elementary tensors are symmetric rank one tensors and projective classes of symmetric rank one tensors are parameterized via the following map.
\begin{defn}
Let $N={{n+d} \choose d}-1$. The \emph{$d$th Veronese embedding} is
\begin{align*}
\nu_d \colon \mathbb{P}^{n} &\longrightarrow \mathbb{P}^N=\mathbb{P}(\sym^d\mathbb{C}^{n+1})\\
[v]&\mapsto [v^{\otimes d}].
\end{align*}
The image of the above map $X_{d}:=\nu_d(\mathbb{P}^{n})$ is the \emph{$d$th Veronese variety}.
\end{defn}

Recall that a projective variety $\subset \PP^n$ is \emph{non-degenerate} if it is not contained in any hyperplane, i.e. $\mathrm{span} \{ X\}=\PP^n$. 

Symmetric rank one tensors are parameterized by Veronese varieties but when dealing with higher rank tensors, it is useful to look at the following auxiliary varieties.

\begin{defn}
Let $X\subset \mathbb{P}^{N}$ be an irreducible, non-degenerate projective variety and fix a positive integer $r$.
The \emph{$r$th secant variety} $\sigma_r(X)$ of $X$ is 
\[
\sigma_r(X)= \overline{ \bigcup_{p_1,\dots,p_r \in X} \mbox{span} \{p_1,\dots,p_r \}}.
\]
\end{defn}
We recall that secant varieties form a chain of subvarieties in which the previous variety is contained
in the subsequential variety until we reach the ambient space:
$$
X=\sigma_1(X) \subset \sigma_2(X)\subset \cdots \subset \sigma_t(X)=\mathbb{P}^N,
$$
where the smallest integer $t$ such that $\sigma_t(X)=\mathbb{P}^N$ is called the \emph{generic rank} of $X$.

From a parameter count one can easily see that
$$
\dim(\sigma_r(X))\leq \min\{ r(\dim X+1)-1, N  \},
$$
where sometimes the above inequality is strict.
\begin{defn}\label{def:defectivity}
	An irreducible non-degenerate variety $X$ is $r$-\emph{defective} if $$
 \dim(\sigma_r(X))<\min\{ r(\dim X+1)-1, N  \}.
	$$
\end{defn}

Let us introduce another variety that will be useful for understanding the symmetric geometric rank.

\begin{defn}Let $X\subset \mathbb{P}^N$ be a projective variety.
The \emph{tangential variety} $\tau(X)$ of $X$ is the Zariski closure of the union of all tangent spaces at smooth points of $X$.
\end{defn}

\begin{rem}\label{rem:tangential_LM^2}
 Let $X_d\subset \mathbb{P}^N$ be the $d$th Veronese variety. It is classically known that any element in $\tau(X_d)$ can be seen as $[L^{d-1}M]$ for some linear forms $L,M$ \cite[Section 1]{CGG02}.
\end{rem}

We recall that the general linear group $\mathrm{GL}_{n+1}(\mathbb{C})$ acts on an elementary symmetric tensor $v^{\otimes d}\in \sym^d\mathbb{C}^{n+1}$ as
$$
g\cdot v^{\otimes d} = (g\cdot v)^{\otimes d},
$$
for $g \in \mathrm{GL}_{n+1}(\mathbb{C})$. By linearity, this action can be extended to any element of $\sym^d\mathbb{C}^{n+1}$. Moreover, the same type of operation can be extended to non invertible matrices. Let $M_{n+1}(\mathbb{C})$ denote the set of $(n + 1) \times (n+1)$ matrices with coefficients in $\CC$. For  $A \in M_{n+1}(\mathbb{C})$ and  $T\in \sym^d\mathbb{C}^{n+1}$
we have
\[A \cdot T := \sum_{i_1,\dots,i_d} t_{i_1,\dots,i_d} Ae_{i_1} \otimes \dots \otimes A e_{i_d}.
\]
This operation induces a partial ordering of symmetric tensors.

\begin{defn}\label{def:leq}
Let $S, T\in \sym^d\mathbb{C}^{n+1}$, we write 
$S\leq T$ if there exists a matrix $A\in M_{n+1}(\mathbb{C})$ such that 
\[
S = A \cdot T.
\]
\end{defn}
\begin{notation}
For every $r \in \mathbb{N}_{>0}$ with $r\leq n$, denote $I_r\in \sym^d\mathbb{C}^{n+1}$ as the tensor $$ I_r=\sum_{i=0}^r (e_i)^{\otimes d}.$$  
	\end{notation}

Since it will be useful later, let us recall other relevant notions of symmetric rank.

\begin{defn}\label{rmk:other_ranks}
Let $T\in (\mathbb{C}^{n+1})^{\otimes d}$ and denote  $T(x_1,\dots,x_d)$ as the description of $T$ as a multilinear polynomial

Let $S\in \sym^d\CC^{n+1}$. 
\begin{enumerate} \setlength\itemsep{.5em}
    \item[$\circ$] The \emph{Geometric rank} of $T$ \cite{kopparty2020geometric} is  $$ \gr(T)=\codim\{ (x_1,\dots,x_{d-1}) \in (\mathbb{C}^{n+1})^{d-1} ~\vert ~ \forall x_{d}\in \mathbb{C}^{n+1}, \  T(x_1,\dots,x_d)=0   \}.$$ 
    \item[$\circ$] The \emph{Symmetric subrank} of  $S$ \cite{CFTZsymsub} is $$\mathrm{Q}_s(S) = \max \{ r ~\vert ~ A\cdot S = I_r \hbox{ for some } A\in M_{n+1}(\CC) \}.  \quad  $$
\end{enumerate}
\end{defn}

We conclude this section with the notion of a normal variety. 
\begin{defn}
Let $X\subset \PP^N$ be a non-degenerate projective variety, so in particular $X$ is covered by affine varieties $X_i=U_i\cap X$ for affine charts $U_i$, $i=0,\dots,N$. We say that $X$ is \emph{normal} if for all $i=0,\dots, N$ the coordinate sheaf $\mathcal{O}(X)$ is integrally closed in $\mathrm{Frac}(\mathcal{O}(X))$. 	
\end{defn}

\begin{thm}[{{\cite[Chapter 5, Theorem 2.19]{Shaf}}}]\label{rem:normal_sing}
If $X $ is an irreducible, normal variety then $\codim (\mathrm{Sing}(X)) \geq 2$ in $X$.
\end{thm}

\section{Properties of symmetric geometric rank}\label{sec:properties}
We now wish to prove analogous properties for symmetric geometric rank as those established for geometric rank in \cite[Section 4]{kopparty2020geometric}. We first prove that the  symmetric geometric rank is monotone.

\begin{lem}\label{lem:monotone} Let $S,T\in \sym^d\mathbb{C}^{n+1}$. If $S \leq T$, then $\sgr(S) \leq \sgr(T)$.
\end{lem}
\begin{proof}
    Let $F(x)$ be the polynomial associated to $T$. Then $F(A^Tx)$ is the polynomial associated to the tensor $S = A \cdot T$.
    First observe that if $A$ is invertible, then $\sgr(S) = \sgr(T)$ since the singular locus of $F(x)$ and $F(A^Tx)$ are birationally equivalent. If $A$ is of rank $r <n+1$, then without loss of generality, we can take $A$ to be the projection onto the first $n+1 - r$ variables $x_1,\ldots,x_{n+1-r}$. This implies that 
    \[ \dim(\text{Sing}(F(A^Tx)) \geq \dim(\text{Sing}(F(x)), \]
giving the result.
\end{proof}

In the following lemma we prove that the symmetric geometric rank is additive under direct sum.

\begin{lem}\label{lem:additive}
Let $T_1, T_2\in \sym^d\mathbb{C}^{n+1}$. Then   $$\sgr(T_1 \oplus T_2) = \sgr(T_1)+\sgr(T_2).$$
\end{lem}
\begin{proof}
The tensor $T = T_1 \oplus T_2$ is a block diagonal tensor with blocks $T_1$, $T_2$. Therefore, if we denote $F_1(x),F_2(y)$ as the polynomials associated to $T_1$ and $T_2$ respectively, then the polynomial associated to $T$ is $F(x,y) = F_1(x) + F_2(y)$ where the variables $x$ and $y$ are disjoint. Suppose $\text{Sing}(F_1)$ is cut out by polynomials $\{f_0,\ldots,f_{n} \} \subset \CC[x_0,\ldots,x_n]$ and $\text{Sing}(F_2)$ is cut out by polynomials $\{g_0,\ldots,g_n\} \subset \CC[y_0,\ldots,y_n]$, then $\text{Sing}(F) $ is cut out by polynomials $\{f_0,\ldots,f_n,g_0,\ldots,g_n\} \subset \CC[x_0,\ldots,x_n,y_0,\ldots,y_n\}$. Therefore, 
\[ \text{codim}(\text{Sing}(F) ) = \text{codim}(\text{Sing}(F_1)) + \text{codim}(\text{Sing}(F_2)). \qedhere\]
\end{proof}

Element wise subadditivity is another property of the symmetric geometric rank.
\begin{lem}\label{lem:subadditive}
	Let $S,T\in \sym^d\mathbb{C}^{n+1}$. Then 
 $\sgr(S+T) \leq \sgr(S)+\sgr(T)$.
\end{lem}
\begin{proof}
Since $S+T \leq S \oplus T$, by \Cref{lem:monotone} we have $\sgr(S+T) \leq \sgr(S \oplus T)$. By \Cref{lem:additive}, $\sgr(S \oplus T) = \sgr(S) + \sgr(T)$, proving the result.
\end{proof}

We now move on to comparing the symmetric geometric rank to other notions of symmetric tensor rank. We begin by comparing the symmetric geometric rank with the geometric rank. Recall that any tensor $T \in (\CC^{n+1})^{\otimes d}$ can be seen as a multilinear polynomial in variable groups $\{x_1,\ldots,x_d\}$ where $x_i = \{x_{i,1},\ldots,x_{i,n+1}\}$.

In this notation, the geometric rank of $T$ is the codimension of the variety obtained by taking the partial derivatives of $F$ with respect to one of the variable groups, namely
\[
\gr(T) = \text{codim} \biggl\{(x_1,\ldots,x_{d-1}) \in \CC^{(d-1)(n+1)} \ \bigg| \ \frac{\partial F}{\partial x_{d,1}} = 0, \ldots, \frac{\partial F}{\partial x_{d,n+1}} = 0 \biggr\}.
\]

\begin{lem}\label{lem:SGR_leq_GR}
Let $T\in \sym^d\mathbb{C}^{n+1}$. Then $\sgr(T) \leq \gr(T)$.
\end{lem}
\begin{proof}
Let $T\in \sym^d\CC^{n+1}$ be such that $\gr(T) = r$ and let $F$ be the multilinear polynomial in the variable groups $\{x_1,\ldots,x_d\}$, $x_i \in \CC^{n+1}$, associated to $T$ and $G$ the usual homogeneous polynomial in the variables $y \in \CC^{n+1}$ associated to $T$. Define, 
\begin{align*}
V_1 &= \biggl \{(x_1,\ldots,x_{d-1},y) \in \CC^{d(n+1)} \ \bigg| \ \frac{\partial F}{\partial x_{d,1}} = 0, \ldots, \frac{\partial F}{\partial x_{d,n+1}} = 0 \biggr \} \\
V_2 &= \biggl \{y \in \CC^{n+1} \ \bigg| \ \frac{\partial G}{\partial y_1} = 0, \ldots, \frac{\partial G}{\partial y_{n+1}} = 0 \biggr\}
\end{align*}
Since $\gr(T) = r$, $\dim(V_1) = d(n+1)-r$. Observe that $V_2$ is in bijection with  $V_1 \cap V_3$ where 
\[ V_3 = \{(x_1,\ldots,x_{d-1},y) \in \CC^{d(n+1)}  \ \vert~ x_{i,1} - y_0 = 0,\ldots, x_{i,n+1} - y_{n+1} = 0, \ \forall 1 \leq i \leq d-1 \}. \]
Note that $\dim(V_3) = n+1$. 
Therefore by the Affine Dimension Theorem, 
\begin{align*}
   \dim(V_1 \cap V_3) &\geq \dim(V_1) + \dim(V_3) -d(n+1)  \\
   &= d(n+1)-r + (n+1) -d(n+1)  \\
   &= n+1 - r
\end{align*} 
giving $\codim(V_2) = \codim (V_1 \cap V_3) \leq r$.
\end{proof}

The following example shows that there exist symmetric tensors for which the inequality in \Cref{lem:SGR_leq_GR} is strict. 
\begin{example}
	Let $T = e_1\otimes e_1\otimes e_2 + e_1 \otimes e_2 \otimes e_1 + e_2 \otimes e_1 \otimes e_1\in \sym^3\mathbb{C}^2$. One can directly compute that $\gr(T)=2$ and 
    $\sgr(T)=1$.	
\end{example}

We now compare the symmetric geometric rank with the symmetric subrank (\Cref{rmk:other_ranks}). To do this, we first need a preliminary lemma.

\begin{lem}\label{lem:sgr identity}
Let $r \in \mathbb{N}_{>0}$ with $r\leq n$. Then $\sgr(I_r) = r$ where $I_r \in \sym^{d}\CC^{n+1}$ is the identity tensor.
\end{lem}
\begin{proof}
	By direct computation, the variety $\{x \in \mathbb{P}^n ~\vert ~ dx_1^{d-1} = 0,\ldots, dx_r^{d-1} = 0 \} $ has codimension $r$.
\end{proof}

\begin{thm}\label{thm:bound_sym_subrank}
For any tensor $T \in \sym^d\CC^{n+1}$ we have $\mathrm{Q}_s(T) \leq \sgr(T)$.
\end{thm}
\begin{proof}
Assume that $\mathrm{Q}_s(T) = r$ so $I_r \leq T$. Then by \Cref{lem:sgr identity}, $r = \sgr(I_r) \leq \sgr(T)$.\end{proof}

\begin{cor}\label{cor:SGR_matrices}
Let $T\in \sym^2\mathbb{C}^{n+1}$ be a symmetric order 2 tensor, i.e. a symmetric matrix. Then $\sgr(T) = \text{rank}(T)$.
\end{cor}
\begin{proof}
One can easily see that for matrices we have $\sgr(T) = \gr(T)$ and \cite[Section 1.2]{kopparty2020geometric} shows that for matrices $\gr(T) = \text{rank}(T)$.
\end{proof}

The above corollary shows that the inequality $\mathrm{Q}_s(T) \leq \sgr(T)$ can be strict. For instance, take the $2n \times 2n$ matrix defined in \cite[Example 2.7]{CFTZsymsub}. Then the symmetric subrank is always $0$, but the symmetric geometric rank will be $2n$.

To summarize the second half of this section, we have that
$$
\mathrm{Q}_s(T) \leq \sgr(T)\leq \gr(T).
$$

\section{Classifying spaces with prescribed symmetric geometric rank}\label{sec:spacesprescribedSGR}
In this section we focus on classifying spaces of tensors with bounded symmetric geometric rank (\Cref{spaces_sgr_r}). As a first example, we consider the case of symmetric matrices, i.e. we describe $\mathcal{S}_{2,r}$. 

\begin{prop}\label{prop:sgr_mats}
For degree two homogeneous polynomials the notion of symmetric geometric rank coincide with the notion of symmetric rank, i.e. 
   $$
   \mathcal{S}_{2,r}=\sigma_r(X_{2}),
   $$  
   where $X_2=\nu_2(\mathbb{P}^n)$ is the $2$nd Veronese variety.
\end{prop}

\begin{proof}
    Let $F\in  \mathrm{Sym}^2\mathbb{C}^{n+1}$, since $F(x)=\frac{1}{2}x^TAx$ for some symmetric matrix $A\in M_{n+1}(\mathbb{C})$, then $\nabla F(x)=Ax$.
    Therefore 
    \begin{equation*}
    \dim (\mathrm{Sing}(F))=\dim \{ [x] \in \mathbb{P}^{n} ~\vert ~ \nabla F(x)=0\}=\dim\{ [x] \in \mathbb{P}^{n} ~\vert ~ Ax=0\},
\end{equation*}
i.e. $\codim (\mathrm{Sing}(F))=\mathrm{rank}(A)$ and hence
$
   \mathcal{S}_{2,r}=\sigma_r(X_{2}).
$
\end{proof}

\begin{rem}
It is well known that for matrices the symmetric rank and rank coincide. Therefore, by \Cref{prop:sgr_mats} we have that for order two tensors, the symmetric rank, rank and symmetric geometric rank are all equal. This provides additional evidence that the notion of symmetric geometric rank is natural to consider.
\end{rem}

Now we focus on describing $\mathcal{S}_{d,r}$ when $d\geq 3$ and $r\geq 1$. The following theorem gives a full geometric description of the spaces $\mathcal{S}_{d,r}$ of symmetric tensors with prescribed symmetric geometric rank.

\begin{theorem}\label{thm:full_S_dr}
Let $d\geq 3$, $r\geq 1$, $n\geq 1$ and for an $r$-plane $\ell \subset \mathbb{P}^n$ set $H_\ell:= \{ [F]\in \mathbb{P}(\sym^d\mathbb{C}^{n+1})^* ~\vert ~ V(F) \mbox{ is tangent to } \ell  \}$ . Then
$$
\mathcal{S}_{d,r}=\bigcap_{\ell\subset \mathbb{P}^n}H_\ell,
$$
where the intersection runs over all linear spaces $\ell \subset \mathbb{P}^n$ of dimension $r$.
\end{theorem}
\begin{proof}First we show that $\mathcal{S}_{d,r}\subseteq \bigcap_{\ell\subset \mathbb{P}^n}H_\ell$.
Let $[F]\in \mathcal{S}_{d,r}$, so $\dim (\mathrm{Sing}(F))\geq n-r$ and hence for any $r$-plane $\ell\subset \mathbb{P}^n$ we can find a point in the intersection $\mathrm{Sing}(F)\cap \ell $ which is a tangent point.

For the other direction, let $[F]\in \bigcap_{\ell \subset \mathbb{P}^n} H_\ell$ and assume by contradiction that $[F]\not\in~\mathcal{S}_{d,r}$. Therefore $\dim( \mathrm{Sing}(F))\leq n-( r+1)$ and the general $r$-plane $\ell $ does not intersect $\mathrm{Sing}(F)$, which means that $V(F)$ and $\ell$ intersect away from the singular locus of $V(F)$. But then by Bertini Theorem we get that $V(F)$ cannot be tangent to $\ell$ which is a contradiction. 
\end{proof}

The above theorem 
gives a complete geometric description of 
the spaces $\mathcal{S}_{d,r}$. We now wish to bound for the dimension of these varieties. In order to do so, notice that for all $r=0,\dots,n$ $$\Bigg[\sum_{i=0}^r x_i^d \Bigg]\in \mathcal{S}_{d,r+1}.$$
This idea can be easily generalized by considering cones, i.e. hypersurfaces in $\mathbb{P}^n$ whose defining equation depends on less than $n+1$ variables.

\begin{rem}\label{rem:secants}
For all $[F]\in \mathbb{P}(\sym^d\mathbb{C}^{n+1})^* $ such that $F=F(x_0,\dots,x_r)$ we have $\codim (\mathrm{Sing}(F))\geq n-(r+1)$. Therefore, we get the following lower bound on the dimension of $\mathcal{S}_{d,r+1}$:
   $$
   \dim (\mathcal{S}_{d,r+1})\geq { d+r \choose r}-1.
   $$
\end{rem}

\subsection{The space $\mathcal{S}_{d,1}$}
We now focus on symmetric tensors with symmetric geometric rank one by first understanding the shape of an element in $\mathcal{S}_{d,1}$.
\begin{lem}\label{lem:Sd1}
	Let $d\geq 2$ and take $[F] \in \mathcal{S}_{d,1}$. Then $F = G^i \cdot H$ with $i \geq 2$, for polynomials $G,H$ such that $i \cdot \deg(G) + \deg(H) = d$.
\end{lem}
\begin{proof}
    Since $\text{Sing}(F)$ must be a hypersurface in $\mathbb{P}^n$, the polynomial defining the hypersurface $\text{Sing}(F)$ must appear with multiplicity in $F$.
\end{proof}

In this particular case of $r=1$, we can prove the specific version of \Cref{thm:full_S_dr} in the following more elementary way.

\begin{thm}\label{thm:S_{3,1}}
 Let $H_\ell = \{ [F] \in \PP (\sym^d \CC^{n+1})^*~\vert ~ V(F)$ is tangent to a line $\ell \subset \PP^n$ \}. Then 
 \begin{equation}\label{eq:space_general_S_1}
 \bigcap_{\ell \subset \PP^n} H_\ell = \mathcal{S}_{d,1}.
 \end{equation}
\end{thm}
\begin{proof}
    First we show $\mathcal{S}_{d,1} \subseteq \bigcap_{\ell \subset \PP^n} H_\ell$. Let $[F]\in \mathcal{S}_{d,1}$. By \Cref{lem:Sd1}, we can write $F = G^i H$ for polynomials $G,H \in \CC[x_0,\ldots,x_n]$ and $i \geq 2$. This implies that $V(G) \subseteq \text{Sing}(F)$ for all lines $\ell \subset \PP^n$. Therefore, there exists a point $p \in \text{Sing}(F) \cap \ell$ with multiplicity greater than or equal to $2$, meaning that $V(F)$ is tangent to $\ell$ at $p$.

    For the reverse inclusion, let $[F] \in \bigcap_{\ell \subset \PP^n} H_\ell$. Note that $\ell \subseteq V(F)$ for all lines $\ell \subset \PP^n$ if and only if $F$ is identically zero, so we can assume $\ell$ is not contained in $V(F)$. This means $\#(\ell \cap V(F))= \{p_1,\ldots,p_d\}$ and, without loss of generality, we can assume $p_1$ has multiplicity greater than or equal to $2$. This implies that $V(F)$ is singular at $p_1$. Since $V(F)$ must be singular at a point for every line $\ell \subset \PP^n$, this implies that $\text{Sing}(F) \cap \ell \neq \emptyset$ for every line $\ell$, therefore the codimension of the singular locus of $V(F)$ must be less than or equal to $1$. 
\end{proof}
We point out that when $d=3$ the space \eqref{eq:space_general_S_1} also appears in problems in enumerative geometry where one wishes to compute the number of cubic hypersurfaces
tangent to a given number of lines \cite{BDFK}.
Moreover, for the particular case of $d=3$ one can see that 
the space $\mathcal{S}_{3,1}$ is actually the tangential variety of the third Veronese variety. 
\begin{thm}
\label{prop:SGR1_tangential}
The space of 3-ways symmetric tensors having symmetric geometric rank one is 
$$
\mathcal{S}_{3,1}=\tau(X_3).
$$
\end{thm}

\begin{proof}
 We first prove that if $[T]\in \tau(X_3)$ then $\sgr(T)=1$. Let $F$ be the form corresponding to a tensor in $\tau(X_3)\setminus X_3$, i.e. $F=L^2M$, where $L=a_0x_0+\cdots +a_{n}x_{n}$ and $M=b_0x_0+\cdots +b_{n}x_{n}$ where vectors $(a_0,\ldots,a_n)$ and $(b_0,\ldots,b_n)$ are not proportional (cf. \Cref{rem:tangential_LM^2}). Notice that $$ \frac{ \partial F}{\partial x_i}=\frac{ \partial L^2M}{\partial x_i}=2a_iLM+b_iL^2,
$$
hence the ideal generating the singular locus of $V(F)$ is given by
$$
( 2a_0LM+b_0L^2,\dots,  2a_{n}LM+b_{n}L^2)=(L^2,LM),
$$
where the second equality follows since the coefficients of $L$ and $M$ are not proportional. The conclusion follows since the radical of the above ideal is generated by~$L$.

To show the reverse inclusion, it suffices to show that any homogeneous cubic polynomial, $F \in \CC[x_0,\ldots,x_{n}]_3$, whose corresponding hypersurface has a singular locus of codimension one, is of the form $F = L^2 M$ where $L = a_0x_0 + \ldots a_{n} x_{n}$ and $M = b_0x_0 + \ldots + b_{n} x_{n}$. Since $V(F)\subset \mathbb{P}^n$ is a hypersurface, for its singular locus to also have codimension one in $\mathbb{P}^n$, $F$ must be reducible. Therefore $F = P \cdot Q$ for homogeneous polynomials $P,Q \in \CC[x_0,\ldots,x_{n}]$. The singular locus of $V(F)$ then contains $V(P) \cap V(Q)$. In other words, if the codimension of the singular locus of $V(F)$ is one, then the codimension of $V(P) \cap V(Q)$ must be one. Therefore, $P$ and $Q$ must contain a common factor. Since $\deg(F) = 3$, the only way for this to happen is if $Q = L$ and $P = L \cdot M$ for linear forms $L,M$. 
\end{proof}

Combining both \Cref{thm:S_{3,1}} and \Cref{prop:SGR1_tangential} allows us to give another interpretation of the tangential variety of the third Veronese variety. This is a straightforward consequence of the aforementioned theorems, 
but for the sake of completeness we include it here.
\begin{cor}The following equality holds for the tangential variety $\tau(X_3) $ of the third Veronese variety $X_3$
    $$
    \tau(X_3)=\bigcap_{\ell \subset \PP^n} H_\ell = \mathcal{S}_{3,1}
$$
where $H_\ell = \{ [F] \in \PP (\sym^d \CC^{n+1})^*~\vert ~ V(F) \hbox{ is tangent to a line } \ell \subset \PP^n\}$ and $\mathcal{S}_{3,1}$ is the space of cubic polynomials with singular locus of codimension one in $\PP^n$. 
\end{cor}
Continuing to work with $d=3$, we can understand the relation between higher secant varieties of the tangential variety and the spaces of tensors with prescribed symmetric geometric rank. 
\begin{prop}\label{prop:SGR_secant_tangential}
	Let $\sigma_s(\tau(X_3))$ be the $r$th secant variety of the tangential variety $\tau(X_3)$ of the third Veronese variety $ X_3\subset  \mathbb{P}^{{{n+3} \choose 3}-1}$. A general tensor $[F] \in \sigma_r (\tau (X_{3}))$ has symmetric geometric rank $r$.
\end{prop}
\begin{proof}
	A general element of $\sigma_r(\tau(X_{3}))$ is of the form $F = \sum_{i=1}^r L_i^2 M_i$, for linear forms $L_i,M_i$. Taking the partial derivatives we have $\sgr(T) \leq r$. Indeed, let $I$ be the ideal generating the singular locus of $F$ i.e. $I = (\frac{\partial F}{\partial x_i})_{i=0}^{n}$. Then $I = ( L_1,\ldots, L_r )$.
\end{proof}
\begin{rem}\label{rem:secanttau_vs_S_r}
The $r$th secant variety of the tangential variety of the third Veronese $X_3$ is defective if and only if $r=n=2,3$ and $4$ (cf. \cite[Thoerem 1.1]{AN18}, also \cite{CGG02}).
Hence, in the non defective cases $$\dim \sigma_r(\tau(X_3))=\min \biggl\{ r(2n+1)-1, {{n+3} \choose 3}-1\biggr\}. $$
Moreover, 
$$
\sigma_r(\tau(X_3))\subset \mathcal{S}_{3,r}
$$
and the containment is strict since $\mathcal{S}_{3,r}$ fills the ambient space for $r=n+1$, while we should set $r\simeq n^2$ to get that $\sigma_r(\tau(X_3))$ fills the ambient space.
\end{rem}

Since we are connecting spaces of symmetric tensors with prescribed symmetric geometric rank with secant varieties, it is worth to also see the interplay between $\mathcal{S}_{3,r}$ and secant varieties of the third Veronese varieties. 
\begin{rem}
Since $X_3\subset \tau(X_3)$ then $\sigma_r(X_3)\subset \sigma_r(\tau(X_d))$ and by \Cref{prop:SGR_secant_tangential} we have $ \sigma_r(X_3)\subset \mathcal{S}_{3,r}$.
\end{rem}

We now completely describe the space $\mathcal{S}_{d,1}$ for all $d \geq 3$. 
\begin{thm}\label{thm:complete_S_k1}
	Fix $d\geq 3$. The space of $d$-factor symmetric tensors with symmetric geometric rank at most 1 can be characterized as follows.
	\begin{enumerate}[label=\arabic*.]\setlength\itemsep{.2em}
		\item If $d=3$ then $\mathcal{S}_{d,1}=\tau(X_3)$ and $\dim (\mathcal{S}_{d,1})=2n$.
		\item Otherwise, $\mathcal{S}_{d,1}$ is reducible and has a component with maximal dimension 
  $$
{{\lfloor \frac{d}{2} \rfloor  +n} \choose \lfloor \frac{d}{2} \rfloor }+ {{d-\lfloor \frac{d}{2} \rfloor+n} \choose {d - \lfloor \frac{d}{2} \rfloor}}-2.
$$
 Moreover, each component is given by the image of the map
  \begin{align*}
   	\phi_{d_1}: \PP (\sym^{d_1}\CC^{n+1})^* &\times  \,\PP (\sym^{d_2}\CC^{n+1})^* \longrightarrow \PP (\sym^{d}\CC^{n+1})^*\\
   	(\,[F]&\,,\,[G]\,) \quad \mapsto \qquad [F^{\,2} \cdot G],
   \end{align*}
   where $d_2 = d - 2d_1$. Finally, all components intersect on the $d$th Veronese $X_d$.
	\end{enumerate}
\end{thm}
\begin{proof}
    Let $d= 3$. We have seen in \Cref{prop:SGR1_tangential} that $\mathcal{S}_{3,1}=\tau(X_3)$, where $X_3$ is the third Veronese variety. It is classically known that $\dim(\tau(X_3))=2n$ and that the tangential variety is irreducible (cf. e.g. \cite[Proposition 1.1]{CGG02}). This covers the case $d=3$.

   Let now $d \geq 4$. By \Cref{lem:Sd1}, given an element $[F]\in \mathcal{S}_{d,1}$, we have that $F = G^i\cdot H$ with $i \geq 2$, for polynomials $G,H$ such that $i \cdot \deg(G) + \deg(H) = d=i\cdot d_1+d_2$. Therefore we can consider the following map \begin{align}\label{eq:map_phi}
   	\phi_{d_1,i}: \PP (\sym^{d_1}\CC^{n+1})^* &\times  \,\PP (\sym^{d_2}\CC^{n+1})^* \longrightarrow \PP (\sym^{d}\CC^{n+1})^*   \\
   	(\,[F]&\,,\,[G]\,) \quad \mapsto \qquad [F^{\,i} \cdot G] \nonumber 
   \end{align}
   and look at the image of this map in order to get different components of $\mathcal{S}_{d,1}$. Notice that if $i>2$ then a form of type $F^{\,i}\cdot G$ can be factored as $F^{\, 2}\cdot (F^{\, i-2}\cdot G)$. Therefore the component given by $\mathrm{Im}(\phi_{d_1,i})$ would be contained in the one given by $\mathrm{Im}(\phi_{d_1,2}):= \mathrm{Im}(\phi_{d_1})$. 
  Since $\phi_{d_1}$ parameterizes each component of the variety $\mathcal{S}_{d,1}$, to find the dimension of a component, it is enough to study the dimension of the image of this map.
 Note that the map $\phi_{d_1}$ is the composition of two maps:
    \begin{align*}
\mathbb{P}(\sym^{d_1}\mathbb{C}^{n+1*})&\rightarrow \mathbb{P}(\sym^{2d_1}\mathbb{C}^{n+1*})\rightarrow \mathbb{P}(\sym^{d}\mathbb{C}^{n+1*})\\
        [F] \quad &\mapsto \qquad [F^{\,2}] \quad \mapsto \quad [F^{\,2}\cdot G],
    \end{align*}
where $[G]\in\mathbb{P}(\sym^{d_2}\CC^{n+1})^* $. The map sending $F\mapsto F^{\,2} $ is 2 to 1 while the map sending $F \to F^2\cdot  G$ is injective. 
Therefore, by the fiber dimension theorem we have $$
\dim (\mathrm{Im}(\phi_{d_1}))={{d_1+n}\choose d_1}+{{d_2+n}\choose d_2}-2.
$$
In particular the maximum dimensional component is achieved for $d_1=\lfloor \frac{d}{2} \rfloor $.
   
    Finally, it suffices to show that for $d\geq 4$, $\mathcal{S}_{d,1}$ is reducible and all components intersect only on the $d$th Veronese variety. First, it is clear that the $d$th Veronese is the intersection of all components. To show $\mathcal{S}_{d,1}$ is not irreducible it suffices to show that all components are not contained in one component. Specifically, it suffices to show that $\mathrm{Im}(\phi_{d_1})$ is not contained in the maximal dimension component for all $d_1$ such that $2 d_1+d-2d_1=d$.  Let $d_1<\lfloor \frac{d}{2} \rfloor $ and take irreducible forms $F,G$ of degrees $d_1$ and $d-2d_1$ respectively. We have that the class of $F^{\,2}\cdot G$ is an element of $\mathrm{Im}(\phi_{d_1})$ but it cannot be an element of $ \mathrm{Im}\left( \phi_{\lfloor \frac{d}{2}\rfloor} \right)$.
\end{proof}

\subsection{The space $\mathcal{S}_{d,2}$} We now consider symmetric tensors having symmetric geometric rank at most two.
Here we distinguish between irreducible hypersurfaces and reducible ones. Looking first at the component given by reducible hypersurfaces gives the following bound on the dimension of $\mathcal{S}_{d,2}$.

\begin{lem}\label{lem: red Sd2} One (possibly reducible) 
component of
	$\mathcal{S}_{d,2}$ is given by reducible polynomials $F = G \cdot H$ where $\deg(G) + \deg(H) = d$. Furthermore, $\dim(\mathcal{S}_{d,2}) \geq \binom{n+d-1}{d-1}+n$.
\end{lem}
\begin{proof}
If $F = G \cdot H$, then $V(G) \cap V(H) \subseteq \text{Sing}(F)$. If $G$ and $H$ have no common factors, then $V(G) \cap V(H)$ has codimension two in $\PP^n$, showing $[F] \in \mathcal{S}_{d,2}$. 

The construction above shows that we can bound the dimension of $\mathcal{S}_{d,2}$ by the space of polynomials $F = G\cdot H$ where $\deg(G) = d_1$, $\deg(H) = d-d_1$. This will achieve the maximum dimension when $d_1 = 1$. In this case, the component of $\mathcal{S}_{d,2}$ corresponding to polynomials with a linear factor has dimension $\binom{n+d-1}{d-1}+n$.
\end{proof}

We now consider irreducible polynomials $[F] \in \mathcal{S}_{d,2}$. By \Cref{rem:normal_sing}, these are non-normal hypersurfaces.

\begin{rem}\label{rem:irr_non-normal}
 Recall that for an irreducible, non-normal variety $X\subset \PP^n$, the singular locus of $X$ is of codimension at most one in $X$ (cf. \Cref{rem:normal_sing}) and hence in the case of non-normal irreducible hypersurfaces $X=V(F)$, we get $[F]\in \mathcal{S}_{d,2}$. Therefore, a component of $\mathrm{S}_{d,2}$ contains irreducible non-normal hypersurfaces.
\end{rem}

We focus now on the case $d=3$, and give a complete description of $\mathcal{S}_{3,2}$.
By \Cref{rem:irr_non-normal} a component of $\mathcal{S}_{3,2}$ is
\begin{align*}
\mathcal{C}_{\mathrm{ir}}=\{ [F]\in \PP(\sym^3\mathbb{C}^{n+1})^* \, \vert \, V(F) \mbox{ is irreducible and non-normal}\}.
\end{align*}
Moreover, by \cite[Lemma 2.4]{WEScubic} if $X=V(F)\subset \mathbb{P}^n$ is a non-normal irreducible cubic hypersurface then the singular locus of $X$ is linear. In \cite{WEScubic} the authors classify non-normal, irreducible, cubic hypersurfaces and they show that if $n\geq 5$, the hypersurface given by an element in $\mathcal{C}_{\mathrm{ir}}$ is a cone  \cite[Remark 2.3]{WEScubic}.

\begin{table}[h]
\begin{tabular}{|l|l|l}
\cline{1-2}
$n$                       & Normal form of cubic polynomials with $\sgr$ 2                                                                          &  \\ \cline{1-2}
\multicolumn{1}{|l|}{2} & \begin{tabular}[c]{@{}l@{}}$x_0^2x_2+x_1^3$\\ $x_0^2x_2+x_1^3+x_1^2x_2$\end{tabular} &  \\ \cline{1-2}
\multicolumn{1}{|l|}{3} & $x_0^2x_2+x_1^3+x_0x_1x_3$                                                           &  \\ \cline{1-2}
\multicolumn{1}{|l|}{4} & $x_0^2x_2+x_1^3+x_1^2x_3+x_0x_1x_4$                                                  &  \\ \cline{1-2}
\end{tabular}
\caption{The normal form of irreducible cubic polynomials with symmetric geometric rank two  in fewer than six variables \cite[Theorem 3.1]{WEScubic}.}\label{table:cubic}
\end{table}

The classification in \Cref{table:cubic} suggests that all irreducible degree three polynomials with symmetric geometric rank two have a common shape given by 
$$L_1^2 \cdot M_1 + L_2^2 \cdot M_2 + L_1 \cdot L_2 \cdot M_3,$$ 
for linear forms $L_1,L_2,M_1,M_2,M_3$. Since the space of non-normal irreducible cubic polynomials is a component of $\mathcal{S}_{3,2}$ we would like to be more precise about this.

\begin{prop}\label{lemma:dimS_2ic}
Let $n \geq 3$ and $F \in \mathcal{C}_{\mathrm{ir}}$. Then,
\[ F = L_1^2 \cdot M_1 + L_2^2 \cdot M_2 + L_1 \cdot L_2 \cdot M_3 \]
for linear forms $L_1,L_2,M_1,M_2,M_3$ such that $L_1,L_2$ are not proportional, $M_1 , M_2 \neq 0$, and $M_1,M_2,M_3$ are not proportional. Moreover, $\dim(\mathcal{C}_{\mathrm{ir}}) = 5n-2$.
\end{prop}
\begin{proof}Let us first show that $\mathcal{C}_{\mathrm{ir}}$ parametrizes all irreducible, cubic polynomials whose singular locus is of codimension two.
By \cite[Lemma 2.4]{WEScubic}, if $[F]\in \mathbb{P}(\sym^3\mathbb{C}^{n+1})^*$ is an irreducible, cubic polynomial such that $V(F)$ has a singular locus of codimension two, then $\mathrm{Sing}(F)$ is linear, which means that we can write $\text{Sing}(F) = V( L_1, L_2 )$ for independent linear forms $L_1,L_2$. Clearly, if $F= L_1^2 \cdot M_1 + L_2^2 \cdot M_2 + L_1\cdot L_2 \cdot M_3$, then the ideal of the singular locus $I_{\mathrm{Sing}(F)}\subseteq(L_1, L_2 )$. Since we know that $\text{Sing}(F)$ has codimension $2$ and it is linear, this means that $\text{Sing}(F) = V(L_1, L_2 )$. 

We now have to show that any cubic hypersurface $V(F)$ such that $\text{Sing}(F) = V( L_1, L_2 )$ for linear polynomials $L_1, L_2$ is of the form
\[ F = L_1^2 \cdot M_1 + L_2^2 \cdot M_2 + L_1\cdot L_2\cdot M_3 . \]

First of all, note that if $L_1,L_2$ are linearly dependent then $[F]\in \mathcal{S}_1$. Without loss of generality, we can consider a change of basis so that $L_1 = x_1$ and $L_2 = x_2$
By Euler's identity, we can write $F= x_1 \cdot G_1 + x_2 \cdot G_2$ for some $G_1,G_2$ with $\deg(G_1)=\deg(G_2)=2$. This gives,
\begin{align*}
    \frac{\partial F}{\partial x_1} &= G_1 + x_1 \cdot \frac{\partial G_1}{\partial x_1} + x_2 \cdot \frac{\partial G_2}{\partial x_1}, \\
    \frac{\partial F}{\partial x_2} &= x_1 \cdot \frac{\partial G_1}{\partial x_2} + G_2 + x_2 \cdot \frac{\partial G_2}{\partial x_2}.
\end{align*}
Since both $\frac{\partial F}{\partial x_1}, \frac{\partial F}{\partial x_2}$ must vanish on $V(L_1,L_2)= V(x_1,x_2)$, we have that
\begin{align*}
    G_1 &= x_1 \cdot P_1 + x_2 \cdot P_2, \\
    G_2 &= x_1\cdot  Q_1 + x_2 \cdot Q_2,
\end{align*}
for some linear forms $P_1,P_2,Q_1,Q_2$. Substituting these expressions into $F$, we have that
\begin{align*}
    F&= x_1 \cdot G_1 + x_2 \cdot G_2 \\
    &= x_1 (x_1\cdot P_1 + x_2\cdot P_2) + x_2 (x_1 \cdot Q_1 + x_2 \cdot Q_2) \\
    &= x_1^2\cdot  P_1 + x_2^2\cdot Q_2 + x_1 x_2 (P_2 + Q_1).
\end{align*}
This shows that $\mathcal{C}_{\mathrm{ir}}$ parametrizes all irreducible cubic hypersurfaces whose singular locus is of codimension two. 

 We conclude this proof by computing the dimension of $\mathcal{C}_{\mathrm{ir}}$.
Consider the following incidence variety
$$I := \{ ([F],\mathrm{Sing}(F)) ~\vert~ [F]\in  \mathcal{C}_{ir}\} \subseteq
\mathbb{P}(\sym^3\mathbb{C}^{n+1}) ^*\times \mathrm{Gr}(n-1,n+1).
$$ Once we fix hyperplanes $\ell_1=V(L_1),\ell_2=V(L_2)$, then for any $M_1,M_2,M_3$ we have that $\mathrm{Sing}(L_1^2 \cdot M_1 + L_2^2 \cdot M_2 + L_1 \cdot L_2 \cdot M_3)=V(L_1,L_2)$. So a general fiber above the second projection from $I$ has dimension $3n$.
Moreover, the image of $I$ via the first projection is $\mathcal{C}_{\mathrm{ir}}$ and so a general fiber above the first projection is
a singleton. Therefore, since the dimension of the Grassmannian is $\dim (\mathrm{Gr}(n-1,n+1))=(n-1)(n+1-n+1)=2n-2$ then $\dim(\mathcal{C}_{\mathrm{ir}})=2n-2+3n=5n-2$. 
\end{proof}

Now that we completely described irreducible, cubic hypersurfaces $V(F)$ for which $\sgr(F)=2$, Let us restate \Cref{lem: red Sd2} for the case  $d = 3$.

\begin{lem}\label{lemma:redS2_linear}
Let $F$ be a reducible cubic not of the form $F = L^2 \cdot M$ for linear forms $L, M$. Then $\sgr(F)= 2$.
\end{lem}

We are now ready to fully describe the space $\mathcal{S}_{3,2}$.

\begin{thm}\label{thm:S_{3,2}}
The class of any element in $\mathcal{S}_{3,2}$ is either of the form \[F = L_1^2 \cdot M_1 + L_2^2 \cdot M_2 + L_1\cdot L_2 \cdot M_3 \]
for  linear forms $L_1,L_2,M_1,M_2,M_3$ satisfying the conditions in \Cref{lemma:dimS_2ic} or it is reducible. Specifically, the space $\mathcal{S}_{3,2} $ can be decomposed as $$\mathcal{S}_{3,2}=  \mathcal{C}_{\mathrm{ir}} \cup \mathcal{C}_{\mathrm{re}},$$ where 
$\mathcal{C}_{\mathrm{ir}}$ consists of polynomials of the above form and
$\mathcal{C}_{\mathrm{re}}$ consists of reducible cubics. Moreover,  
$$\dim (\mathcal{C}_{\mathrm{re}})=\binom{n+2}{2} +n \hbox{ and } \dim(\mathcal{C}_{\mathrm{ir}})=5n-2,$$ and when $n \geq 6$ the component $ \mathcal{C}_{\mathrm{re}}$ is of larger dimension.
\end{thm}
\begin{proof}
First note that $\mathcal{S}_{3,2}$ is reducible, so $\mathcal{S}_{3,2} = \mathcal{C}_{\mathrm{ir}} \cup\mathcal{C}_{\mathrm{re}}$ and the intersection of the two components is precisely $\mathcal{S}_{3,1}$. For $F \in \mathcal{C}_{\mathrm{re}} \backslash \mathcal{S}_{3,1}$, $F = G \cdot H $, where $\deg(G) = 2$, $\deg(H) = 1$ and $H \nmid G$. By \Cref{lemma:redS2_linear}, $\text{Sing}(F) = V( G, H)$ is codimension two. To compute the dimension of $\mathcal{C}_{\mathrm{re}}$, observe that reducible cubics can be parameterized via
\begin{align*}
	\psi : \PP(\sym^2\CC^{n+1})^*&\times \PP(\sym^1\CC^{n+1})^*\longrightarrow \PP(\sym^3\CC^{n+1})^*\\
	(\,[G] \,,& \,[H]\,) \quad \mapsto \quad  [G\cdot H]
\end{align*}
and the image of $\psi$ gives us the dimension of $\mathcal{C}_{\mathrm{re}}$. Hence $\dim (\mathcal{C}_{\mathrm{re}})=\binom{n+2}{2} +n$. For the dimension and the description of the irreducible component we refer to \Cref{lemma:dimS_2ic}.
\end{proof}

The above result completely characterize $\mathcal{S}_{3,2}$ and it ends our detailed discussion on the case $(d,r)=(3,2)$. While we leave a complete description of $\mathcal{S}_{d,r}$ as future work, one can still give information on the space  $\mathcal{S}_{d,r}$ for large $d$. 
\begin{rem}
Observe that in \Cref{thm:S_{3,2}} the space $\mathcal{S}_{3,2}$ decomposes into reducible and irreducible polynomials. For the irreducible polynomials, these are precisely polynomials that are singular along a codimension two linear subspace.  By \cite{Slavov2015TheMS}, for sufficiently large $d$, the largest irreducible component of the space $\mathcal{S}_{d,r}$ will also consist of polynomials singular along a codimension $r$ linear space. Therefore, making a construction analogous to the one used for computing the dimension of $\mathcal{C}_{\text{ir}}$ in the proof of \Cref{lemma:dimS_2ic} can be used to approximate the largest dimensional component of $\mathcal{S}_{d,r}$ so long as $d$ is sufficiently large.
\end{rem}

\section{Some concrete examples}
We conclude this paper by computing the symmetric geometric rank of  relevant symmetric tensors as well as providing concrete descriptions of $\mathcal{S}_{d,r}$ for small $n$.

\subsection{Computing the symmetric geometric rank of relevant tensors}

Let us start with a very standard example of symmetric tensor, i.e. the fully symmetric tensor. 
\begin{example}[Fully symmetric tensor]
	Let $T\in \sym^3\mathbb{C}^{n+1}$be the fully symmetric tensor with $n\geq 2$, i.e.
	$$
	T=\sum_{i,j,k} e_{i}\otimes e_{j}\otimes e_{k}.
	$$
	The polynomial representation $h_3$  of $T$ is given by the sum of all degree three monomials in the variables $x_0,\dots,x_n$, namely
	$$
	h_3=\sum_{0\leq i\leq j \leq k\leq n}x_ix_jx_k=x_0^3+x_0^2x_1+\cdots+x_{n-1}x_n^2+x_n^3. 
	$$
Let us prove that $V(h_3)$ is smooth and hence $\sgr(T)$ is maximal. We want to argue by induction on $n\geq 2$.
Before proving our result, let us notice that the complete homogeneous polynomial of degree two $h_2(x_0,\dots,x_n) = \sum_{0 \leq i \leq j \leq n} x_ix_j$ is smooth.
Solving $\partial h_2(x) / \partial x_i=0$ for all $i=0,\dots,n$ corresponds to solving the linear system $Ax=0$ where 
$$
A=\begin{pmatrix}
2 & 1 &1 & \dots & 1\\
1 & 2 & 1 & \dots & 1\\
\vdots& & \ddots & &\vdots\\
1& & &2 & 1\\
1&1 &\dots &1 & 2
\end{pmatrix} 
$$
Since $A$ is invertible, the only possible solution is the trivial one and hence $V(h_2)$ is smooth.

Let us start now the induction. For the base case $n=2$ one can directly compute that $V(h_3)$ is smooth. So let $n\geq 3$. Notice that for all $i=0,\dots,n$ the polynomial $h_3$ can be written as 
$$
h_3=x_ih_2(x_0,\dots,x_n)+P_i(x_0,\dots, x_n)
$$
where $h_2(x_0,\dots,x_n)$ is the complete homogeneous polynomial of degree two and $P_i$ does not depend on $x_i$. Therefore, for all $i=0,\dots,n$
\begin{align}\label{eq:sys_h3}
\frac{\partial h_3}{\partial x_i}= h_2(x_0,\dots,x_n)+ x_i \frac{\partial h_2(x_0,\dots,x_n)}{\partial x_i}.
\end{align}
Moreover, by Euler's formula 
\begin{align*}
\frac{\partial h_3}{\partial x_i}= \sum_{k=0}^n x_k  \frac{\partial h_2(x_0,\dots,x_n)}{\partial x_k}+ x_i \frac{\partial h_2(x_0,\dots,x_n)}{\partial x_i}.
\end{align*}  Having $ \partial h_3/\partial x_i =0$ for all $i $ corresponds to solving the system
$$
A=\begin{pmatrix}
2 & 1 &1 & \dots & 1\\
1 & 2 & 1 & \dots & 1\\
\vdots& & \ddots & &\vdots\\
1& & &2 & 1\\
1&1 &\dots &1 & 2
\end{pmatrix} 
\begin{pmatrix} 
x_0 \frac{\partial h_2}{\partial x_0} \\
\vdots \\
\vdots \\
x_n \frac{\partial h_2}{\partial x_n}
\end{pmatrix} = \mathbf{0}.
$$
Let $x$ be a solution of this system. Since $A$ is invertible, we have $x_i \frac{\partial h_2}{\partial x_i}=0 $ for all $i$ and hence we have that $h_2(x) = 0$.
Therefore, $x$ is also a solution for the system given by $\{h_2=0, \partial h_2 / \partial x_i=0\}$.
Since we earlier proved that $h_2$ is smooth, the constraints
$$
\{h_2=0, \ x_i \partial h_2 / \partial x_i=0\}
$$
 enforces that at least one of the $x_i$ must be zero. Without loss of generality we may assume $x_n=0$. But now the tuple $(x_0,\dots,x_{n-1})$ is a common zero for $h_3(x_0,\dots,x_{n-1})$ and all its partial derivatives, so we can conclude by induction that $x_i = 0$ for all $i$ and hence $h_3$ is smooth.
\end{example}

The \emph{Coppersmith-Winograd tensor} is a class of tensors introduced in \cite{CoWite}. These tensors have been helpful for giving bounds on the exponent characterizing the complexity of matrix multiplication (cf. \cite{Landcompltheory}).
\begin{example}[Big Coppersmith-Winograd tensor]\label{ex:bigCW}Fix  a positive integer $q<n$ and let $T\in (\CC^{n+1})^{\otimes 3}$ be the tensor
	{\small
	\begin{align*}
	T&= \sum_{i=1}^q (e_0\otimes e_i\otimes e_i +e_i\otimes e_0\otimes e_i+e_i\otimes e_i\otimes e_0)\\
	&\qquad \quad + e_0\otimes e_0\otimes e_{q+1}+e_0\otimes e_{q+1}\otimes e_0+e_{q+1}\otimes e_0\otimes e_{0}\\
	&= W_1+\cdots +W_{q}+\tilde{W},
	\end{align*}
}
\noindent where we denote $W_i=e_0\otimes e_i\otimes e_i +e_i\otimes e_0\otimes e_i+e_i\otimes e_i\otimes e_0$ and $\tilde{W}=e_0\otimes e_0\otimes e_{q+1}+e_0\otimes e_{q+1}\otimes e_0+e_{q+1}\otimes e_0\otimes e_{0}$. We recognize that these are all representations of the $W$-state tensor which is an important tensor in the quantum information literature (cf. \cite{Wref}). Clearly $T\in \sym^3\CC^{n+1}$ and therefore it makes sense to compute its symmetric geometric rank. Looking at $T $ as a homogeneous degree three polynomial $F$ we have 
$$
F=\sum_{i=1}^q x_0x_i^2+x_0^2x_{q+1}=x_0(x_1^2+\cdots +x_q^2+x_0x_{q+1})
$$
and by direct computation one sees that $\sgr(F)=\codim (\mathrm{Sing}(F))=2$. 
 \end{example}
 
 \begin{example}[Small Coppersmith-Winograd tensor]\label{ex:smallCW}
 	Fix  a positive integer $q<n$ and let $T\in (\CC^{n+1})^{\otimes 3}$ be the tensor
 		\begin{align*}
 			T=& \sum_{i=1}^q (e_0\otimes e_i\otimes e_i +e_i\otimes e_0\otimes e_i+e_i\otimes e_i\otimes e_0)= W_1+\cdots +W_{q},
 		\end{align*}
where, as before, we denote $W_i=e_0\otimes e_i\otimes e_i +e_i\otimes e_0\otimes e_i+e_i\otimes e_i\otimes e_0$. The polynomial representation of $T$ is
$$
F=\sum_{i=1}^q x_0x_i^2=x_0(x_1^2+\cdots +x_q^2)
$$
 and, as before, one can directly compute that $\sgr(T)=2$. 
 	\end{example}
 	
 Another interesting example is the \emph{Maximal compressibility tensor} (cf. \cite{LM17}).
 \begin{example}[Maximal compressibility tensor]\label{ex:maxc} Fix  a positive integer $q\leq n$ and let $T\in (\CC^{n+1})^{\otimes 3}$ be the tensor
 	\begin{align*}
 		T=& e_0\otimes e_0\otimes e_0+\sum_{i=1}^q (e_0\otimes e_i\otimes e_i +e_i\otimes e_0\otimes e_i+e_i\otimes e_i\otimes e_0).
 	\end{align*}
 	The polynomial representation of $T$ is
 	$$
 	F=x_0^3+ \sum_{i=1}^q x_0x_i^2=x_0(x_0^2+x_1^2+\cdots +x_q^2),
 	$$
 	so $\sgr(T)=2$.
 	\end{example}
  \begin{rem}
The geometric rank of \Cref{ex:bigCW}, \Cref{ex:smallCW} and \Cref{ex:maxc} is computed in \cite[Examples 5.5, 5.6 and 5.8]{GL22} respectively. In all the three cases the geometric rank is three while the symmetric geometric rank is two. 
\end{rem}

 The rest of this part is devoted to computing the symmetric geometric rank of the symmetrized part of the matrix multiplication tensor, namely
 $$
 sM_{\langle n+1 \rangle }(A,B,C):=\frac{1}{2}[\mathrm{trace}(ABC)+\mathrm{trace}(BAC)]\in \sym^3\CC^{(n+1)^2},
 $$
where $A,B,C$ are matrices of size $n+1$. In \cite[Proposition 2.3]{CHILO}, the authors proved that  the singular locus of $\{sM_{\langle n+1\rangle}=0\}$ is the variety
$$
\mathrm{Sing} (sM_{\langle n+1\rangle})=\{  [A]\in \PP M_{n+1}(\CC) ~\vert ~ A^2=0 \}.
$$
Therefore, to determine the symmetric geometric rank, it suffices to find the dimension of this variety.

\begin{prop}
  Let $n\geq 1$, then $$
  \dim \{  [A]\in \PP M_{n+1}(\CC) ~\vert ~ A^2=0 \}=\frac{(n+1)^2}{2}-1.
  $$
\end{prop}
\begin{proof}
Denote $X=\{  [A]\in \PP M_{n+1}(\CC) ~\vert ~ A^2=0 \}$.
    Notice that $[A]\in X$ if and only if $\mathrm{Im}(A)\subset \mathrm{Ker}(A) $. For any matrix $A$, denote $\mathrm{Ker}(A)=K_A$ and call
    $$ \Sigma=\{ (K_A,[f_A]) ~\vert ~ f_A: K_A^\perp \rightarrow K_A \mbox{ linear} \}.  $$
Notice that the map
\begin{align*}
\Sigma \qquad  &\xrightarrow{\hspace*{1cm}} \qquad X  \\
(K_A,[f_A]) &\mapsto \begin{pmatrix} [A]: K_A\oplus K_A^\perp \xrightarrow{\hspace*{.4cm}} K_A\oplus K_A^\perp\\
 k\in K_A \mapsto 0, \,
        k' \in K_A^\perp \mapsto f_A(k')
\end{pmatrix}
\end{align*}
   is an isomorphism, so $\dim (X)=\dim (\Sigma)$.  The space $\Sigma$ can be decomposed as $\Sigma=\cup_{n\geq 0} \Sigma_k$ where $$
    \Sigma_k=\{ (K_A,[f_A]) ~\vert ~ f_A: K_A^\perp \rightarrow K_A \mbox{ linear}, \, \dim K_A=k \}
    $$
    and now we can compute the dimension of $\Sigma_k$ by way of the Grasmannian $\mathrm{Gr}(k,n+1)$ Consider the projection
    \begin{align*}
      \varphi : \Sigma_k& \, \xrightarrow{\hspace*{.9cm}}  \, \mathrm{Gr}(k,n+1)\\
        (K_A,&[f_A])\mapsto K_A.
    \end{align*}
  A general fiber of $\varphi$ is given by all linear maps $f_A: K_A^\perp \rightarrow K_A$, so it is of dimension $k(n+1-k)-1 $ and hence $ \dim (\Sigma_k) =2k(n+1-k)-1$. Then $\dim (\Sigma)=\max \{ \dim \Sigma_k \}$ which occurs for $k=(n+1)/2$. Thus $\dim (X)=(n+1)^2/2-1$. 
    \end{proof}

\begin{cor}\label{cor:sgr_sym_matr_mult}
    The symmetric geometric rank of the symmetrized part $sM_{\langle n+1\rangle}$ of the matrix multiplication tensor is $\sgr(sM_{\langle n+1\rangle})=\frac{(n+1)^2}{2}$.
\end{cor}

\subsection{Examples of the tensor space stratification }
We conclude this section with a detailed study of spaces of prescribed symmetric geometric rank in the case of $n= 1,2$ for $d=3$. For both cases we find equations of all $\mathcal{S}_{3,r}$ contained in $\mathcal{S}_{3,n+1}=\PP(\sym^3\CC^{n+1})$.

Let us start with the classical case of binary cubics.
\begin{example}[$n=1,d=3$]
	Let $F \in (\mathrm{Sym}^3\mathbb{C}^2)^*$, so
	\[F = a_1 x^3 + a_2 x^2y + a_3 xy^2 + a_4 y^3 \]
	where $a_1,\ldots,a_4 \in \CC$.
	We wish to look at spaces of bounded symmetric geometric rank in $\PP( \mathrm{Sym}^3\CC^2)$: $$\mathcal{S}_{3,2}=\PP(\mathrm{Sym}^3\CC^2) \supset  \mathcal{S}_{3,1} .$$
 The space	$\mathcal{S}_{3,1} \subset \PP (\mathrm{Sym}^3\CC^2)$ is the set of tensors with symmetric geometric rank one. This is equivalent to the set of polynomials in $\CC[x,y]_3$ with a zero-dimensional singular locus. By \Cref{thm:S_{3,1}} we have $\mathcal{S}_{3,1}=\tau(X_3)$, hence we get the defining equations for this space by using that $F$ has symmetric geometric rank $1$ if and only if $F= L^2 M$ for linear forms $L, M \in \CC[x,y]_1$. Using this parameterization and equating coefficients, we have that the  variety $\mathcal{S}_{3,1}$ is $2$ dimensional since is is defined as the vanishing of a non-zero polynomial, namely
	\[ \mathcal{S}_{3,1} = V(a_2^2a_3^2-4a_1a_3^3-4a_2^3a_4+18a_1a_2a_3a_4-27a_1^2a_4^2 )  . \]
	A visual of the space $\mathcal{S}_{3,1}$ on the affine chart $a_4 = 1$ is shown in \Cref{fig:s31}. 
	Finally, the space $\mathcal{S}_{3,2}\setminus \mathcal{S}_{3,1}$ is the set of symmetric tensors with symmetric geometric rank two. This is equivalent to the set of smooth cubics in two variables, so it can be defined as $\mathbb{P}(\mathrm{Sym}^3\CC^2)\backslash \mathcal{S}_{3,1}$.
\end{example}
\begin{figure}[h!]
	    \centering
	    \includegraphics[width = 0.4\textwidth]{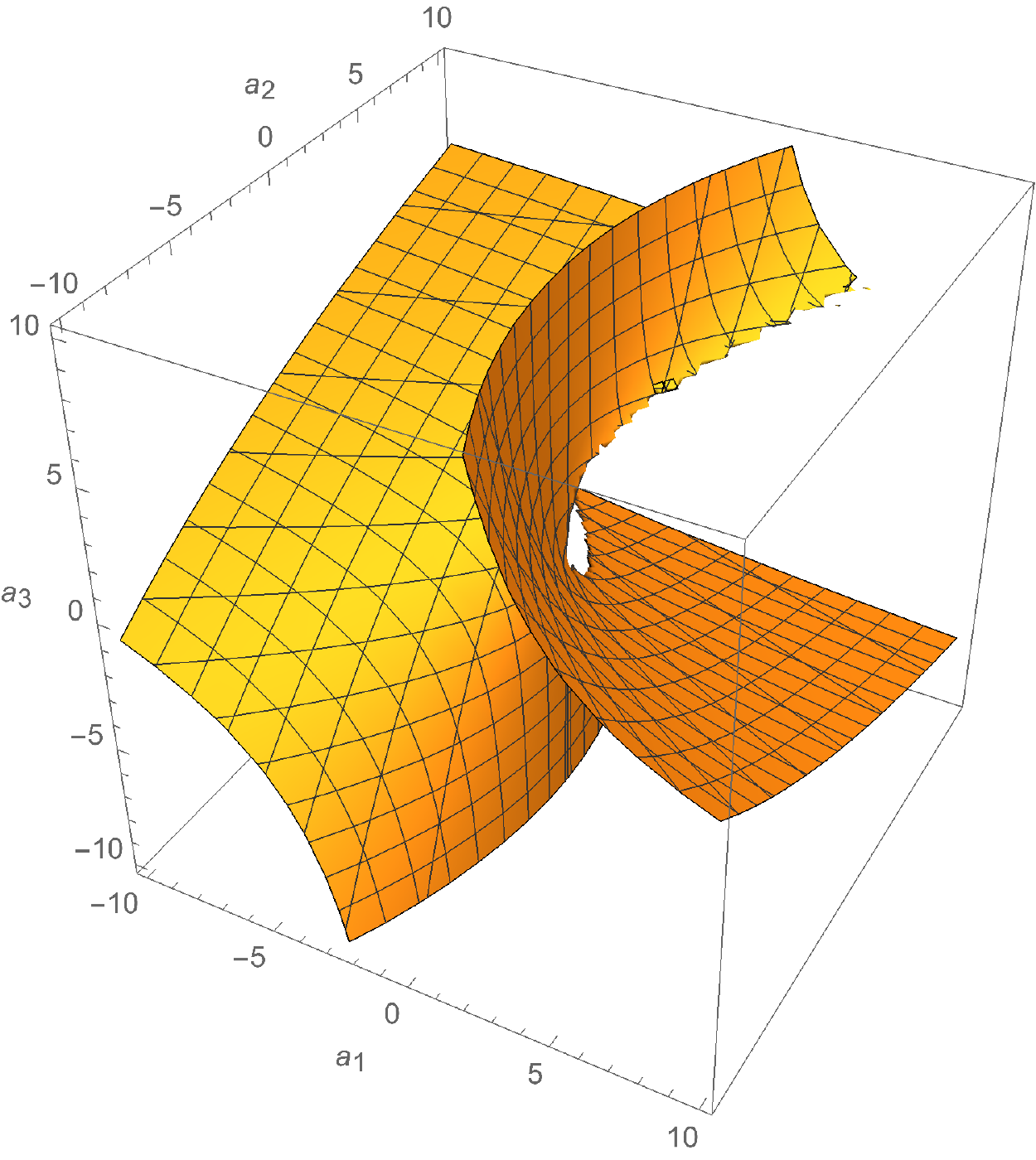}
	    \caption{The space $\mathcal{S}_{3,1}$when $a_4 =1$.}
	    \label{fig:s31}
	\end{figure}

Let us pass now to the case of cubic curves.
\begin{example}[$n=2,d=3$]
	A symmetric $3 \times 3 \times 3$ tensor can be identified with the polynomial 
	\[F = a_1 x^3 + a_2 x^2 y + a_3 x^2 z + a_4xy^2 + a_5xz^2 + a_6xyz + a_7y^3 + a_8y^2z + a_9yz^2 + a_{10}z^3, \]
	where $a_1,\ldots,a_{10} \in \CC$. Again, we wish to look at spaces of bounded symmetric geometric rank in $\PP(\mathrm{Sym}^3\CC^3)$:  $$\mathcal{S}_{3,3}=\PP(\mathrm{Sym}^3\CC^3) \supset  \mathcal{S}_{3,2} \supset \mathcal{S}_{3,1}.$$
	As above, $\mathcal{S}_{3,1}=\tau(X_3)$ (cf. \Cref{thm:S_{3,1}}) and  we are able to get defining equations for $\mathcal{S}_{3,1}$ by using the parameterization $F = L^2 M$ for linear polynomials $L,M$ and equating coefficients. The space $\mathcal{S}_{3,1}$ is defined as the vanishing of $68$ polynomials in $a_1,\ldots,a_{10}$. There are $20$ polynomials of degree $3$, $31$ of degree $4$, $13$ of degree $5$ and $4$ of degree $6$. These polynomials cut out a $4$ dimensional variety in $\PP(\mathrm{Sym}^3\CC^3)$.

	The space of polynomials with symmetric geometric rank at most $2$ is  $\mathcal{S}_{3,2}\subset \PP(\mathrm{Sym}^3\CC^3)$ are precisely homogeneous cubics in $3$ variables with a  singularity. The set of such cubics are defined by the vanishing of the \emph{discriminant}, which is an irreducible polynomial in the variables $a_1,\ldots,a_{10}$. The discriminant is a classical object in algebraic geometry that has been studied in many contexts. For a complete presentation of discriminants, we recommend \cite{GKZ}. In this case, the discriminant has $2040$ monomials of total degree $12$ in the coefficients
	$a_1,\ldots,a_{10}$. It is of degrees $(4,6,6,6,6,9,4,6,6,4)$
	in $a_1,\ldots,a_{10}$ respectively. Since $\mathcal{S}_{3,2}$ is defined as the vanishing of a non-zero polynomial, $\dim(\mathcal{S}_{3,2}) = 8$.
\end{example}

\bibliographystyle{plain}
\bibliography{refs}

\end{document}